\documentclass[9pt]{article}
\usepackage[dvips,final]{graphicx}
\usepackage{amsthm}
\usepackage{amssymb}
\usepackage{amsfonts}
\usepackage[displaymath,mathlines]{lineno}
\usepackage{newlfont}
\usepackage[centertags]{amsmath}
\usepackage[margin=1in]{geometry}
\newcommand{\Rd}{\mathbb{R}^d}
\newcommand{\norm}[1]{\left\Vert#1\right\Vert}

\newcommand{\abs}[1]{\left\vert#1\right\vert}
\newcommand{\seq}[1]{\left<#1\right>}
\newcommand{\set}[1]{\left\{#1\right\}}
\newcommand{\Real}{\mathbb R}

\theoremstyle{plain}
\newtheorem{thm}{Theorem}[section]

\newtheorem{lem}[thm]{Lemma}
\newtheorem{prop}[thm]{Proposition}
\theoremstyle{definition}
\newtheorem{defn}{Definition}[section]
\newtheorem{rem}{Remark}[section]
\title{The catalytic Ornstein-Uhlenbeck Process with Superprocess Catalyst}%

\author{Juan-Manuel Perez-Abarca\\  (jperez@unpa.edu.mx)\\Universidad del Papaloapan ( Mexico )\\Donald A. Dawson  (don.dawson@rogers.com)\\ Carleton University (Canada )}%

\begin{document}
%
\maketitle

 \begin{abstract}
The main objective of this work is to study a natural class of catalytic Ornstein-Uhlenbeck (O-U)
processes with a measure-valued random catalyst, for example, super-Brownian motion. We relate this
to the class of affine processes that provides  a unified setting in which to view
Ornstein-Uhlenbeck processes, superprocesses, and  Ornstein-Uhlenbeck processes with superprocess
catalyst.  We then review some basic properties of super-Brownian motion which we need and
introduce the Ornstein-Uhlenbeck process with catalyst given by a superprocess. The main results
are the affine characterization of the characteristic functional-Laplace transform of the joint
catalytic O-U process and catalyst process and the identification of basic properties of the
quenched and annealed versions of these processes.
\end{abstract}
\vspace{1cm}
 \textbf{Key concepts:} Stochastic partial differential equations,  superprocesses, measure-valued processes, affine processes, catalytic
Ornstein-Uhlenbeck processes, moment measures, characteristic Laplace functional, quenched and
annealed processes in a random medium. \vspace{1cm}

\def\baselinestretch{1.66}

\section{Introduction}

 Beginning with the  work of K. It\^o there has been extensive study of the class of infinite dimensional Ornstein-Uhlenbeck processes (e.g. \cite{Ito-1}).
 In the finite dimensional case O-U processes belong to the family of affine processes  (e.g. Duffie et al \cite {Dus:1}).  Other affine processes
such Cox-Ingersoll-Ross and Heston processes arise in financial modelling and a  characterization
of the general finite dimensional affine processes is known (cf.  \cite {K:R}).
  The main topic of this paper is the notion of catalytic
infinite dimensional O-U processes which give examples of infinite dimensional affine process. It
was established in \cite{PAD 12} that catalytic O-U processes can arise as fluctuation limits of
super-Brownian motion in a super-Brownian catalytic medium. These processes involve measure-valued
catalysts and the resulting catalytic O-U processes have as state spaces a class of Sobolev spaces.
The annealed versions also  give examples of infinite dimensional non-Gaussian random fields.

\section{The catalytic Ornstein-Uhlenbeck}

The name Ornstein-Uhlenbeck process, was  originally given to the process described by the
stochastic differential equation:

$$dX_t=\theta(a-X_t) dt + \sigma dB_t$$
where $\theta>0, a\in\mathbb{R},\sigma>0$ are parameters, $X_t\in\mathbb{\Real}$  and $B_t$ is
Brownian motion. The corresponding infinite dimensional analogue has developed into what is now
known as the \emph{generalized Ornstein-Uhlenbeck process} : \begin{equation*} dX_t=AX_t dt +
dW_t.\end{equation*} Here $X_t$ takes values in some Hilbert space $H$; $A\in\mathcal{L}(H)$ and
$\{W_t\}_{t\geq 0}$ is a Hilbert-space-valued Wiener process. One important case is the cylindrical
Wiener process whose distributional derivative:
$\frac{\partial^2W_t}{\partial t \partial x}$
is space-time white noise.
\\
\smallskip
\\

In this paper, we will consider the \emph{ Generalized Ornstein-Uhlenbeck (OU) process in catalytic
media}, that is, a process that satisfies a stochastic evolution equation of the form:
\begin{equation}\label{E1} dX(x,t)=AX(x,t)+dW_{\mu}(x,t) \end{equation} where $A$ and $X$ are as before, but
this time, $\mu=\{\mu_t\}_{t\geq 0}$ is a measure-valued function of time and $W_{\mu}$ is a Wiener
process based on  $\mu(t)$, i.e. $W$ defines a random set function such that for  sets
$A\in\mathcal{E} $:
\begin{itemize}
\item[(i)] $W_{\mu}(A\times[0,t])$ is a random variable with law $ \mathcal{N}(0,\int_0^t\mu_s(A))ds.$
\item[(ii)]  if $A\cap B =\phi$ then $W_\mu(A\times [0,t])$ and $W_\mu(B\times [0,t])$ are independent and
$W_\mu((A \cup B)\times [0,t])= W_\mu(A\times [0,t])+W_\mu(B\times [0,t])$.
\end{itemize}
As it will be seen later, $\mu_t$ will play the role of the \underline{catalyst}.   For the rest of
the discussion, we will assume that $\mu_t$ is a measure-valued Markov process, for example, the
super-Brownian motion (SBM).
\medskip

As a simple example, consider the  case of a randomly moving atom $\mu_t=\delta_{B_t}$ where $B_t$
is a Brownian motion in $\mathbb{R}^d$ starting at the origin.  In the case of a random catalyst
there are two processes to consider. The first is the solution of the perturbed heat equation
conditioned on a given realization of the catalyst process - this is called the \textit{quenched}
case.  The second is the process with probability law  obtained by averaging the laws of the
perturbed heat equation with respect to the law of the catalytic process - this is called the
\textit{annealed} case.

We will now determine the behavior of the annealed process and
show that it also depends  on the dimension as  expressed in the
following:

\medskip

\begin{thm}
 Let  $X(t,x)$ be the solution of (\ref{E1}) where $\mu_t=\delta_{B_t}$ with $B_t$  a Brownian motion in $\mathbb{R}^d$ and with $A=\frac{1}{2}\Delta$. Then
 $X$ is given by:
 \begin{equation}\label{1:pert}
X(t,x)=\int_0^t\;\;\int_{\Rd} p(t-s,x,y)W_{\delta_{B(s)}}(dy,ds)
\end{equation}
 then, the annealed variance of
  $X(t,x)$ is given by:
 \begin{gather*}
   E\left[\text{Var }X(t,x)\right] =     \begin{cases}
         1/4    & d=1,\quad x=0 \\
         <\infty &d=1,\quad x\ne 0\\
         \infty   & d\geq2
    \end{cases}
\end{gather*}
 \end{thm}

 \begin{proof}
The second moments are computed as follows:
 \begin{equation*}
 \mathbb{E}X^2(t,x)=\mathbb{E}\int_0^t\frac{1}{2\pi(t-s)}
 \exp\left(-\frac{\|x-B(s)\|^2} {(t-s)}\right) ds.
 \end{equation*}
In the case $d=1$, $x=0$ the expectation in the last integral can be computed using the
 Laplace transform $M_X$ of the $\chi^2_1$ distribution as:
 $$ \int_0^t\frac{1}{2\pi(t-s)}\mathbb{E}\exp\left(-\frac{B^2(s)}{t-s}
 \right)ds$$
 with:
 \begin{equation*}
 \begin{split}
 \mathbb{E}\exp\left(-\frac{B^2(s)}{t-s}\right)&=\frac{1}{2\pi}\mathbb{E}
 \exp\left(-\frac{s}{t-s}\frac{B^2(s)}{s}\right)\\
 &=\frac{1}{2\pi}M_X\left(-\frac{s}{t-s}\right)=\frac{1}{2\pi}\left(\frac{t-s}{t+s}\right)^{1/2}
 \end{split}
 \end{equation*}
 A trigonometric substitution shows:
 $$\mathbb{E}X^2(t,0)=\frac{1}{2\pi}\int_0^t\frac{1}{(t^2-s^2)^{1/2}}ds
 = \frac{1}{4}.$$

 For $x\neq0$ and $d=1$, using a spatial shift we have
\begin{equation*}
 \begin{split}
 \mathbb{E}X^2(t,x)&=\mathbb{E}\int_0^t\frac{1}{2\pi(t-s)}
 \exp\left(-\frac{(B(s)-x)^2} {(t-s)}\right) ds\\
 &=\int_0^t\frac{1}{2\pi(t-s)(2\pi s)^{1/2}}\int
 e^{-\frac{(y-x)^2}{t-s}}e^{-\frac{y^2}{2s}}dyds\\
 &\leq \int_0^t\frac{1}{2\pi(t-s)(2\pi s)^{1/2}}\int
 e^{-\frac{(y-x)^2}{t-s}}dyds \\
 &\leq C \int_0^t\frac{1}{(s(t-s))^{1/2}}ds<\infty.
 \end{split}
 \end{equation*}

 \medskip

  When $d\geq2$ and the perturbation is $\delta_{B(t)}$ as before, then
 $\frac{\|B(s)^2\|}{s}$ is distributed as $\chi_d^2$
, and its Laplace transform is:
 $$ M_x(t)=\left[\frac{1}{1-2t}\right]^{d/2}$$
 So: $$\mathbb{E}X^2(t,0)=\int_0^t \frac{ds}{(t^2-s^2)^{d/2}}$$ Which is
 infinite when $d\geq 2$ at $x=0$. A modified calculation  show this is true everywhere.
 \end{proof}

 \begin{rem} Note that the annealed process is not Gaussian since a
similar calculation can show that $E(X^4(t,x))\neq3(E(X^2(t,x)))^2$.
\end{rem}

\section{Affine Processes and Semigroups }
In recent years affine processes have raised a lot of interest, due to their rich mathematical
structure, as well as to their wide range of applications in branching processes,
Ornstein-Uhlenbeck processes and mathematical finance.

\bigskip
In a general   \textit{affine processes} are the class of stochastic processes for which, the
logarithm of the characteristic function of its transition semigroup has the form
$\seq{x,\psi(t,u)}+\phi(t,u)$.
\medskip

Important finite dimensional examples of such processes are the
following SDE's:
\begin{itemize}
\item \textbf{Ornstein-Uhlenbeck process:} $z(t)$ satisfies the
Langevin type equation
$$dz(t)=(b-\beta z(t))dt+\sqrt{2}\sigma dB(t)$$
This is also known as a  Vasieck model for interest rates in
mathematical finance.
 \item \textbf{Continuous state branching
immigration process:} $y(t)\geq 0$ satisfies
$$dy(t)=(b-\beta z(t))dt+\sigma\sqrt{2y(t)}dB(t)$$
with  branching rate $\sigma^2$, linear decay rate $\beta$ and immigration rate $b$. This is also
called the Cox-Ingersoll-Ross (CIR) model in mathematical finance.
\item \textbf{The Heston Model (HM):}  Also of interest in mathematical finance, it assumes that $S_t$ the price of an asset is given by the stochastic differential equation:
$$dS_t=\mu S_t dt+ \sqrt{V_t} S_t dB_t^1 $$
where, in turn, $V_t$  the instantaneous volatility, is a CIR process determined by:
$$dV_t=\kappa(\theta-V_t)dt+\sigma\sqrt{V_t}(\rho dB_t^1+\sqrt{1-\rho^2 }dB_t^2) $$
and $ dB_t^1,  dB_t^2$ are Brownian motions  with correlation  $\rho$,  $\theta$ is the long-run
mean,  $\kappa$ the rate of reversion and $\sigma$ the variance.
\item \textbf{A continuous affine diffusion process in $\mathbb{R}^2$:}
$r(t)=a_1y(t)+a_2z(t)+b$ where
\begin{equation*}
\begin{split}
dy(t)&=(b_1-\beta_{11}y(t))dt\\
&+\sigma_{11}\sqrt{2y(t)}dB_1(t)+\sigma_{12}\sqrt{2y(t)}dB_2(t)\\
dz(t)&=(b_2-\beta_{21}y(t)-\beta_{22}z(t))dt\\
&+\sigma_{21}\sqrt{2y(t)}dB_1(t)+\sigma_{22}\sqrt{2y(t)}dB_2(t)
+\sqrt{2}a dB_0(t)
\end{split}
\end{equation*}
\end{itemize}
where $B_0,B_1,B_2$ are independent Brownian motions.

 It can be seen that the general affine
semigroup can be constructed as the convolution of a \textit{homogeneous} semigroup ( one in which
$\phi=0$) with a \textit{skew convolution semigroup} which corresponds to the constant term
$\phi(t,u)$.
\begin{defn}
A transition semigroup $(Q(t)_{t\geq0})$ with state space $D$ is called a \textit{homogeneous
affine semigroup (HA-semigroup)}  if for each $t\geq0$ there exists a continuous complex-valued
function $\psi(t,\cdot):=(\psi_1(t,\cdot),\psi_2(t,\cdot))$ on $U=\mathbb{C}_-\times(i\mathbb{R})$
with $\mathbb{C}_-=\{a+ib:a\in\mathbb{R}_-,b\in\mathbb{R}\}$ such that:
\begin{equation}\label{4:ha}
\int_D\exp\{\seq{u,\xi}\}Q(t,x,d\xi)=\exp\{\seq{x,\psi(t,u)}\},\qquad
x\in D,u\in U.
\end{equation}
 The HA-semigroup $(Q(t)_{t\geq0})$
given above is  \textit{regular}, that is, it  is stochastically continuous and the derivative
$\psi_t^\prime(0,u)$ exists for all $u\in U$ and is continuous at $u=0$ (see \cite{K:R}).
\end{defn}
\begin{defn}
A transition semigroup $(P(t))_{t\geq 0}$ on $D$ is called a \textit{(general) affine semigroup}
with the HA-semigroup $(Q(t))_{t\geq 0}$ if its characteristic function has the representation
\begin{equation}\label{4:ha2}\int_D\exp\set{\seq{u,\xi}}P(t,x,d\xi)=\exp\set{\seq{x,\psi(t,u)}+ \phi(t,u)},\quad x\in
D,u\in U\end{equation} where $\psi(t,\cdot)$ is given in the above definition and $\phi(t,\cdot)$
is a continuous function on $U$ satisfying $\phi(t,0)=0$.
\end{defn}
Further properties of the affine processes can be found in
(~\cite{Dus:1}, ~\cite{Affn:1} and ~\cite{K:R} ).  We will see that the class of
catalytic Ornstein-Uhlenbeck processes we introduce in the next section
forms a new class of infinite dimensional affine processes.

\section{A brief review on super-processes and their properties}
\smallskip

Given a measure $\mu$ on $\Rd$ and $f\in \mathcal{B}(\Rd)$, denote
by $\seq{\mu,f} :=\int_{\Rd}f\,\text{d}\mu$, let $M_F(\Rd)$ be the
set of finite measures on $\Rd$ and let $C_0(\Rd)_+$ denote the
continuous and positive functions, we also define:
\begin{align}
C_p(\Rd)&=\set{f\in C(\Rd):
\norm{f(x)\cdot\abs{x}^p}_\infty<\infty,p>0}\nonumber\\
M_p(\Rd)&=\set{\mu\in M(\Rd): (1+\abs{x}^p)^{-1}\;d\mu(x) \text{ is a
 finite measure}}\nonumber
\end{align}

\smallskip
\begin{defn}
The $(\alpha,d,\beta)$-superprocess $Z_t$ is the measure-valued
 process, whose Laplace functional is given by:
\begin{equation*}
\mathbb{E}_{\mu}[\text{exp}(-\seq{\psi,Z_t})]=\text{exp}[-\seq{\text{U}_t\psi,\mu}]
 \qquad\qquad\mu\in M_p(\Rd),\,\psi\in C_p(\Rd)_+
 \end{equation*}
where $\mu=Z_0,\text{ and }\text{U}_t$ is the nonlinear continuous
 semigroup given by the  the mild solution of the evolution equation:
\begin{equation}
\begin{split}
\dot{u}(t)&=\Delta_\alpha\,u(t)-u(t)^{1+\beta},\qquad0<\alpha\leq2,\quad
0<\beta\leq 1\\
u(0)&=\psi,\qquad\quad\psi\in D(\Delta_\alpha)_+.
\end{split}
\end{equation}
here $\Delta_\alpha$ is the generator of the $\alpha$-symmetric
stable process, given by:
\begin{eqnarray*}
&\Delta_{\alpha}u(x)=A(d,\alpha)\int_{\mathbb{R}^d}\frac{u(x+y)-u(x)}
{\abs{y}^{d+2\alpha}}dy\\
&A(d,\alpha)=\pi^{2\alpha-d/2}\Gamma((d-2\alpha)/2)/\Gamma(\alpha)
\end{eqnarray*}
  Then, $u(t)$ satisfies the
non-linear integral equation:
\begin{equation*}
u(t)=\text{U}_t\psi-\int_0^t\text{U}_{t-s}(u^{1+\beta}(s))ds.
\end{equation*}
\end{defn}
This class of measure-valued processes was first introduced in ~\cite{Ext:1}, an up-to-date
exposition of these processes is given in \cite{li}. It can be verified that the process with the
above Laplace functional is a finite measure-valued Markov process with sample paths in
$D(\Real_+,M_p(\Rd))$. The special case $\alpha=2,\;\beta=1$ is called super-Brownian motion (SBM)
and this has sample paths in $C(\mathbb{R}_+,M_p(\mathbb{R}^d))$

 Consider the following differential operator on
${M}_p(\mathbb{R}^d)$:
\begin{equation}\label{L:1}
LF(\mu)=\frac{1}{2}\int_{\mathbb{R}^d}\mu(dx)\frac{\delta^2F}{\delta\mu(x)^2}
+\int_{\mathbb{R}^d}\mu(dx)\Delta\left(\frac{\delta
F}{\delta\mu}\right)(x)
\end{equation}
Here the differentiation of $F$ is defined by
\begin{equation*}
\frac{\delta F}{\delta\mu(x)}=\lim\limits_{\epsilon\downarrow 0}
(F(\mu+\epsilon\delta_x)-F(\mu))/\epsilon
\end{equation*}
where $\delta_x$ denotes the Dirac measure at $x$. The domain
$\mathcal{D}(L)$ of $L$ will be chosen a class containing such
functions $F(\mu)=f(\seq{\mu,\phi_1},\dots,\seq{\mu,\phi_1})$ with
smooth functions $\phi_1,\dots,\phi_n$ defined on $\mathbb{R}^d$
having compact support and a bounded smooth function $f\text{ on
}\mathbb{R}^d$. As usual $\seq{\mu,\phi}:=\int\phi(x)\mu(dx)$.\\

The super-Brownian motion  is also characterized as the unique solution to the martingale problem
given by $(L,\mathcal{D}(L))$. The process is defined on the probability space
$(\Omega,\mathcal{F},P_\mu,\{Z_t\}_{t\geq0})$ and
$$ P_\mu(Z(0)=\mu)=1,\quad  P_\mu(Z\in C([0,\infty),M_p(\mathbb{R}^d)))=1.$$  \\
An important property of super-Brownian motion is the compact support property discovered by Iscoe
\cite{Isc:2}, that is, the closed support $S(Z_t)$ is compact if $S(Z_0)$ is compact.
\section{Catalytic OU with the ($\alpha,\text{d},\beta$)-superprocess as catalyst}
\subsection{Formulation of the process}

The main object of this section is the catalytic OU process given
by the solution of
\begin{equation}\label{SPCOU}
dX(t,x)=\frac{1}{2}\Delta X(t,x)dt+W_{Z_t}(dt,dx),\quad X_0(t,x)\equiv 0
\end{equation}
where  $Z_t$ is the $(\alpha,d,\beta)$-superprocess, in the
following discussion, we will assume that the catalyst and the white
noise are independent processes.

In order to study this process we first note that conditioned on the process $\{Z_t\}$, the process $X$ is Gaussian. We next determine the second
moment structure of this Gaussian process.

\begin{prop}
The variance of the process X(t) given by (\ref{SPCOU}), is:
$$\mathbb{E}X^2(t,x)=\int_0^t\!\!\!\int_{\Rd}p^2(t-s,x,u)Z_s(du)ds$$
and its covariance by:
$$\mathbb{E}X(t,x)X(t,y)=\int_0^t\!\!\!\int_{\Rd}p(t-s,x,u)p(t-s,y,u)
Z_s(du)ds$$
\end{prop}

\begin{proof}
In order to compute the  covariance of X(t); recall that the solution
 of (\ref{SPCOU}) is given by the stochastic convolution:
$$X(t,x)=\int_0^t\!\!\!\int_{\Rd}p(t-s,x,u)W_{Z_s}(ds,du)$$
From which, the covariance is computed as:
$$\mathbb{E}X^2(t,x)=\int_0^t\!\!\!\int_{\Rd}\!\!\!\int_0^t\!\!\!\int_{\Rd}
 p(t-r,x,u)p(t-s,x,w)\mathbb{E}[W_{Z_r}(dr,du)W_{Z_s}(ds,dw)]$$
The \underline{covariance measure} is defined by:
\begin{eqnarray}
\textmd{Cov}_{W_{Z}}(dr,ds;du,dw)&\doteq&\mathbb{E}[W_{Z_r}(dr,du)W_{Z_s}(ds,
dw)]\nonumber\\
&=&\delta_s(r)\;dr\;\delta_u(w)\;Z_r(dw)\nonumber
\end{eqnarray}
The second equality is due to the fact that $W_{Z_r}$ is a white
noise perturbation and has the property of independent increments
in time and space.
\end{proof}

\medskip

\subsection{Affine structure of the catalytic OU-process}

We will compute now the characteristic functional of the annealed process $\mathbb{E}[\exp(i
\seq{\phi,X_t})]$, where by definition:
$$\seq{\phi,X_t}\doteq\int_{\Rd}X(t,x)\phi(x)\,\text{d}x,\qquad
\phi\in C(\Rd)$$ which is well-defined since $\text{X}_t$ is a
random field and also the characteristic functional-Laplace functional of the joint process $\{X(t),Z(t)\}$.  For this, we need the following definitions and
properties, for further details see \cite{Oct:1}.

\begin{defn}
Given the $(\alpha,d,\beta)$-superprocess $Z_t$, we define the
weighted occupation time
 process $Y_t$ by
 $$<\psi,Y_t>=\int_0^t<\psi,Z_s>ds\;\;\;\psi\in C_p(\Rd).$$
\end{defn}

\begin{rem}
The definition coincides with the intuitive
interpretation of $Y_t$ as the measure-valued process satisfying:
$$Y_t(B)=\int_0^tZ_s(B)ds,\;\;\;\;\textmd{for }B\in\mathcal{B}(\mathbb{R}^d)$$
\end{rem}
\begin{thm}\label{1:Iscoe}
It can be shown (\!\!~\cite{Oct:1}) that, given $\mu\in M_p(\Rd)$
and $\phi,\psi\in C_p(\Rd)_+$, and $p<d+\alpha$ then the joint
process $[Z_t,Y_t] $ has the following Laplace functional:
$$E_{\mu}[\exp(-<\psi,Z_t>-<\phi,Y_t>)]=\exp[-<U_t^{\phi}\psi,\mu>],\;\;
t\geq0,$$ where $U_t^\phi$ is the strongly continuous semigroup
associated with the evolution equation:
\begin{equation}
\begin{split}
\dot{u}(t)&=\Delta_\alpha u(t)-u(t)^{1+\beta}+\phi\\
u(0)&=\psi
\end{split}
\end{equation}
\end{thm}

\bigskip

A similar expression will be derived when $\phi$ is a function of time,
but before, we need the following:

\begin{defn}A deterministic non-autonomous Cauchy problem, is given by:
\[ (\text{NACP})\quad \quad \quad \left\lbrace
 \begin{array}{ll}\label{nacp}
\dot{u}(t)=A(t)u(t)+f(t)\quad\quad\quad 0\leq s\leq t\\
u(s)=x
           \end{array}
         \right. \]
where $A(t)$ is a linear operator which depends on t.
\end{defn}
\medskip
Similar to the autonomous case, the solution is given in terms of
a
 two parameter family of operators $U(t,s)$, which is called the propagator or the
evolution system of the problem (NACP), with the following properties:
\medskip
\begin{itemize}
\item $U(t,t)=I,\quad U(t,r)U(r,s)=U(t,s)$.
\item $(t,s)\rightarrow U(t,s)\text{ is strongly continuous for }0\leq s\
\leq t\leq T$.
\item The solution of (NACP) is given by:
\begin{equation}
u(t)=U(t,s)x+\int_s^tU(t,r)f(r)dr
\end{equation}
\item in the autonomous case the propagator is equivalent to the
semigroup $U(t,s)=T_{t-s}$.
\end{itemize}

\begin{thm}\label{INHOMEQN}
Let $\mu\in M_p(\Rd),\text{ and } \Phi:\mathbb{R}_+^1\rightarrow
\text{C}_p(\Rd)_+$ be right continuous and piecewise continuous such
 that for each $t>0$ there is a $k>0$ such that
  $\Phi(s)\leq k\cdot(1+\abs{x}^p)^{-1} $ for all $s\in[0,t]$.
   Then
$$\mathbb{E}_\mu\left[\exp\left(-\seq{\Psi,Z_t}-\int_0^t
\seq{\Phi(s),Z_s}\,ds\right)\right]=\exp\left(-\seq{U_{t,t_0}^\Phi\Psi,\mu}
\right)$$where $U_{t,t_0}^\Phi$ is the non-linear propagator
generated by
 the operator $Au(t)=\Delta_\alpha u(t)-u^{1+\beta}(t)+\Phi(t)$, that is,
  $u(t)=U_{t,t_0}^\Phi\Psi$ satisfies the evolution equation:
\begin{equation}
\begin{split}
\dot{u}(s)&=\Delta_\alpha u(s)-u(s)^{1+\beta}+\Phi(s),\qquad t_0\leq s\leq t\\
u(t_0)&=\Psi \geq 0.
\end{split}
\end{equation}
\end{thm}

\medskip

\begin{proof}
The existence of $U_{t,t_0}^\Phi$ follows from the fact that $\Phi$
is a Lipschitz perturbation of the maximal monotone non-linear
operator $\Delta_\alpha (\cdot)-u^{1+\beta}(\cdot)$ (refer to
~\cite{Br} pp. 27).

\medskip

Note that the solution $u(t)$ depends continuously on $\Phi$, to see
this, denote by $V_t^\alpha\text{ the subgroup generated by }
\Delta_\alpha$, then the solution $u(t)$ can be written as:

\medskip
\begin{equation*}
u(t)=V_t(\Psi)+\int_0^t V_{t-s}^\alpha(\Phi(s)-u^{1+\beta}(s))\,ds
\end{equation*}
\medskip
from which the continuity follows.
\medskip

So we will assume first that $\Phi$ is a step function defined on
the partition of $[0,t]$ given by $0=s_0<s_1<\cdots<s_{N-1}<s_N=t$
and $\Phi(t)=\phi_k\text{ on }[s_{k-1},s_k],\quad k=1,\cdots,N$,
we will denote the step function by $\Phi_N$.

\medskip
Note also that the integral:
$$\seq{\Phi(t),Y_t}=\int_0^t\seq{\Phi(s),Z_s}\,ds,\qquad
\Phi:\mathbb{R}_+^1\rightarrow C_p(\mathbb{R}^d)_+$$
can be taken a.s. in the sense of Riemann since we have enough regularity
on the paths of $Z_t$.

\medskip

Taking a Riemann sum approximation:

\begin{equation*}
\begin{split}
&\mathbb{E}_\mu\left(\exp \left[-\seq{\Psi,Z_t}-\int_0^t\seq{\Phi(s),Z_s}
\,ds\right]\right) \\
&=\lim_{N\rightarrow\infty}\mathbb{E}_\mu \left( \exp \left[-\sum_{k=1}
^{N-1}\seq{\phi_k,Z_{\frac{k}{N}t}}\frac{t}{N}-\seq{\frac{t}{N}\phi_N+\Psi,Z_t}
\right]\right).
\end{split}
\end{equation*}

Denote by $U_t^\Phi\text{ and }U_t^{\Phi_N}$ the non-linear
semigroups generated by $\Delta_\alpha
u(t)-u^{1+\beta}(t)+\Phi(t)\text{ and } \Delta_\alpha
u(t)-u^{1+\beta}(t)+\Phi_N(t)$ respectively. Conditioning and using
the Markov property of $Z_t$ we can calculate:

\begin{equation*}
\begin{split}
&\mathbb{E}_\mu\left(\exp \left[-\seq{\Psi,Z_t}-\int_0^t\seq{\Phi(s),Z_s}
\,ds\right]\right) \\
&=\lim_{N\rightarrow\infty}\mathbb{E}_\mu\left(\exp
\left[-\seq{\Psi,Z_t}-\sum_{k=1}^{N}
\int_{s_{k-1}}^{s_k}\seq{\phi_k,Z_s}\,ds\right]\right)\\
&=\lim_{N\rightarrow\infty}\mathbb{E}_\mu\left(\mathbb{E}_\mu\left(
\exp\left[-\seq{\Psi,Z_t}-\sum_{k=1}^{N-1}\int_{s_{k-1}}^{s_k}\!\!\!\!\!
\seq{\phi_k,Z_s}\,ds
-\int_{s_{N-1}}^{s_N}\!\!\!\!\!\seq{\phi_{N-1},Z_s}\,ds
\right]\right.\right.\\
&\left.\left.{}\qquad\qquad\qquad\Large{|}\,\,\sigma(Z_s) ,0\leq
s\leq s_{N-1}\right)\right)
\end{split}
\end{equation*}
\begin{equation*}
\begin{split}
&=\lim_{N\rightarrow\infty}\mathbb{E}_\mu\left(\exp\left(-\sum_{k=1}
^{N-1}\int_{s_{k-1}}^{s_k}\!\!\!\!\!\seq{\phi_k,Z_s}\,ds\right)
\mathbb{E}_\mu\left[\exp(-\seq{\Psi,Z_t}-\int_{s_{N-1}}^t
\!\!\!\!\!\seq{\phi_N,Z_s}\,ds)\right]\right.\\
&\left. {}\qquad\qquad\qquad\LARGE{|}\,\,\sigma(Z_s)
,0\leq s\leq s_{N-1}\right)\\
\end{split}
\end{equation*}
\begin{equation*}
\begin{split}
&=\lim_{N\rightarrow\infty}\mathbb{E}_\mu\left(\exp\left(-\sum_{k=1}
^{N-1}\int_{s_{k-1}}^{s_k}\!\!\!\!\!\!\!\!\seq{\phi_k,Z_s}\,ds\right)
\mathbb{E}_{Z_{s_{N-1}}}\left[\exp(-\seq{\Psi,Z_t}-\!\!\int_{s_{N-1}}^t
\!\!\!\!\!\!\!\!\seq{\phi_N,Z_s}\,ds)\right]\right)\\
&=\lim_{N\rightarrow\infty}\mathbb{E}_\mu\left(\exp\left(-\sum_{k=1}
^{N-1}\int_{s_{k-1}}^{s_k}\!\!\!\!\!\!\!\!\seq{\phi_k,Z_s}\,ds\right)
\exp(-\seq{U_{s_N-s_{N-1}}^{\Phi_N}\Psi,Z_{s_{N-1}}})\right)\\
&=\lim_{N\rightarrow\infty}\mathbb{E}_\mu\left(\exp
\left[-\seq{U_{s_N-s_{N-1}}^{\Phi_N}\Psi,Z_t}-\sum_{k=1}^{N-1}
\int_{s_{k-1}}^{s_k}\seq{\phi_k,Z_s}\,ds\right]\right)
\end{split}
\end{equation*}

\medskip

where $U_{s-s_{N-1}}^{\Phi_N}\Psi$ is the mild solution of the
equation:
\begin{equation*}
\begin{split}
\dot{u}(s)&=\Delta_\alpha u(s)-u(s)^{1+\beta}+\phi_{N-1},\qquad s_{N-1}\leq
s\leq s_N\\u(s_{N-1})&=\Psi.
\end{split}
\end{equation*}

\medskip

Performing this step $N$ times, one obtains:
\begin{equation*}
\begin{split}
&\mathbb{E}_\mu\left(\exp \left[-\seq{\Psi,Z_t}-\int_0^t\seq{\Phi(s),Z_s}
\,ds\right]\right) \\
&=\lim_{N\rightarrow\infty}\mathbb{E}_\mu\left(\exp
\left[-\seq{U_{s_1-s_0}^{\Phi_N}\cdots U_{s_{N-1}-s_{N-2}}^{\Phi_N} U_{s_N-s_{N-1}}
^{\Phi_N}\Psi,Z_{s_0}}\right]\right)\\
&=\lim_{N\rightarrow\infty}\mathbb{E}_\mu\left(\exp
\left[-\seq{U_t^{\Phi_N}\Psi,Z_0}\right]\right)\\
&=\lim_{N\rightarrow\infty}\exp
\left[-\seq{U_0^{\Phi_N}U_t^{\Phi_N}\Psi,\mu}\right]\\
&=\exp\left[-\seq{U_t^\Phi\Psi,\mu}\right]
\end{split}
\end{equation*}
where in the fourth line $U_0^{\Phi_N}=\mathbb{I}$. The last
equality follows because the solution depends continuously on $\Phi$
as noted above and the result follows by taking the limit as
$N\rightarrow\infty$.

\end{proof}
Now let $X_t$ be the solution of:
\begin{equation}\label{SPCOU3}
dX(t,x)=\frac{1}{2}\Delta X(t,x)dt+W_{Z_t}(dt,dx)\quad X(0,x)\equiv 0
\end{equation}
with values in the space of Schwartz distributions on $\mathbb{R}^d$.
 The previous result  will allow us to compute the characteristic-Laplace
functional $\mathbb{E}[\textmd{exp}(i\seq{\phi,X_t}-\seq{\lambda,
Z_t})]$ of the pair $[X_t,Z_t]$.
\medskip
\medskip
\begin{thm}
The characteristic-Laplace functional of the joint process $[X_t,Z_t]$ is given by :
\begin{equation}\label{CFE}\mathbb{E}_\mu[\exp(i\seq{\phi,X_t}-\seq{\lambda,Z_t})] =\exp(-\seq{u(t,\phi,\lambda),\mu})\end{equation} where $u(t,\phi,\lambda)$
is the solution of the equation: \begin{equation}
\begin{split}\label{2:redif}
&\frac{\partial u(s,x)}{\partial s}=\Delta_\alpha
u(s,x)-u^{1+\beta}(s,x)+G^2_\phi(t-s,x),\qquad 0\leq s\leq t\\
&u(0,x)=\lambda
\end{split}
\end{equation}
with $G_\phi$ defined as:
\begin{equation*}
G_\phi(t,s,z)=\int_{\mathbb{R}^d}p(t-s,x,z)\phi(x)dx
\end{equation*}
\end{thm}
\begin{proof} Denote by $\seq{\cdot,\cdot}$ the standard inner product of $L^2$ and  recalling  the following
 property of the Gaussian processes:
$$\mathbb{E}[\exp(i\seq{\phi,X_t})]=\exp(-\textmd{Var}\seq{\phi,X_t})$$
We get
$$[\seq{\phi,X_t}^2]=\int_{\Rd}\!\!\int_{\Rd}X(t,x)X(t,y)\phi(x)\phi(y)\;dx\;dy$$
Hence: \begin{eqnarray}
\text{Var}[\seq{\phi,X_t}]&=&\mathbb{E}[\seq{\phi,X_t}^2]\nonumber
=\int_{\Rd}\!\!\int_{\Rd}\phi(x)\mathbb{E}[X(t,x)X(t,y)]\phi(y)\;dx\;dy\nonumber\\
&=&\int_{\Rd}\!\!\int_{\Rd}\phi(x)\Gamma_t(x,y)\phi(y)\;dx\;dy\nonumber
\end{eqnarray}
where by definition:
$$\Gamma_t(x,y)=\mathbb{E}[X(t,x)X(t,y)]\nonumber
=\int_0^t\!\!\!\int_{\Rd}p(t-s,x,z)p(t-s,y,z)\,Z_s(dz)\,ds\nonumber$$
So:
\begin{eqnarray}
\textmd{Var}[\seq{\phi,X_t}|Z_t]&=&\int_0^t\!\!\!\int_{\Rd}\!\!
\int_{\Rd}\!\!\int_{\Rd}\phi(x)p(t-s,x,z)p(t-s,y,z)
\phi(y)\,Z_s(dz)\,dx\,dy\,ds\nonumber\\
&=&\int_0^t\!\!\!\int_{\Rd}g_\phi(t-s,z)\,Z_s(dz)\,ds\nonumber
\end{eqnarray}
In the last line:
\begin{eqnarray}
g_\phi(t-s,z)&\doteq&\int_{\Rd}\!\!\int_{\Rd}p(t-s,x,z)p(t-s,y,z)
\phi(x)\phi(y)\,dx\,dy\nonumber\\
&=&\left[\int_{\Rd}p(t-s,x,z)\phi(x)\,dx\right]^2\nonumber
\end{eqnarray}
Note that the function:
$$G_\phi(t-s,z)\doteq\int_{\Rd}p(t-s,x,z)\phi(x)\,dx$$
considered as a function of s, satisfies the backward heat equation:
$$\frac{\partial}{\partial s}G_\phi(t-s,z)+\frac{1}{2}\Delta G_\phi(t-s,z)=0,\qquad
0\leq s\leq t$$ with final condition:
$$G_\phi(t)=\phi(x)$$
So:
$$\mathbb{E}[\exp(i\seq{\phi,X_t})]=\exp\left[-\int_0^t<G^2_\phi(t-s,z)
,Z_s(dz)>\,ds\right]$$ where $G^2_\phi$ is given by the above
expression. With this, assuming $Z_0=\mu$ we have the following
expression for the Laplacian of the joint process $[X_t,Z_t]$:
\begin{align}
&\mathbb{E}_\mu[\exp(i\seq{\phi,X_t}-\seq{\lambda,Z_t})]\nonumber\\
&=\mathbb{E}_\mu[\mathbb{E}_\mu[\exp(i\seq{\phi,X_t}-\seq{\lambda,Z_t})|\sigma(Z_s),0\leq s\leq t]]\nonumber\\
&=\mathbb{E}_\mu[\exp(-\seq{\lambda,Z_t})\mathbb{E}_\mu[\exp(i\seq{\phi,X_t})
|\sigma(Z_s),0\leq s\leq t]]&(\text{measurability})\nonumber\\
&=\mathbb{E}_\mu\left[\exp\left(-\int_0^t\int_{\Rd}G^2_\phi(t-s,z)Z_s(dz)\,ds-\seq{\lambda,Z_t}\right)\right]\nonumber\\
&=\mathbb{E}_\mu\left[\exp\left(-\int_0^t\seq{G^2_\phi(t-s,z),Z_s}\,ds
-\lambda\seq{1,Z_t}\right)\right]&(\text{by definition})\nonumber\\
&=\exp(-\seq{u(t),\mu})&(\text{  by Theorem  } \ref{INHOMEQN}
)\nonumber
\end{align}
and the result follows  after renaming the variables.
\end{proof}

\begin{rem}
We can generalize this to include a second term
\begin{equation}\label{SPCOU2}
dX(t,x)=\frac{1}{2}\Delta X(t,x)dt+W_{Z_t}(dt,dx)+b(t,x)W_2(dt,dx),\quad X(0,x)=\chi(x)
\end{equation}
where  $Z_t$ is the $(\alpha,d,\beta)$-superprocess, and $W_2(dt,dx)$ is space-time white noise. In
this case  the characteristic-Laplace functional has the form
\begin{eqnarray}\label{GCFE} &&\mathbb{E}_{\mu,\chi}[\exp(i\seq{\phi,X_t}-\seq{\lambda,Z_t})]\\
&& =\exp\left(-\seq{u(t,\phi,\lambda),\mu}- \int\int q(t,x_1,x_2)\phi(x_1)\phi(x_2)dx_1dx_2 +i\int
\int p(t,x,y)\chi(x)\phi(y)dxdy)\rangle\right)\nonumber.\end{eqnarray} This is an infinite
dimensional analogue of the general affine property defined in (\ref{4:ha2}).
\end{rem}
\begin{rem}
If in Theorem 5.4, $\mathbb{R}^d$ is replaced by a finite set $E$ and $\frac{1}{2}\Delta$ and
$\Delta_\alpha$ are replaced by the generators of Markov chains on $E$, then the analogous
characterization of the characteristic-Laplace functional remains true and describes a class of
finite dimensional (multivariate) affine processes.

\end{rem}

\section{Some properties of the quenched and annealed catalytic O-U processes}
\label{S:properties}
In this section we formulate some basic properties of the catalytic O-U process in the case in
which the catalyst is a super-Brownian motion, that is $\alpha =2$ and $\beta=1$. For processes in
a random catalytic medium $\set{Z_s:0\leq s\leq t}$ a distinction has to be made between the
\textit{quenched} result above, which gives the process conditioned on $Z$ and the
\textit{annealed} case in which  the  process is a compound stochastic process.  In the quenched
case  the process is Gaussian.  The corresponding annealed case leads to a \textit{ non-Gaussian}
process. These two cases  will of course require  different formulations. For the quenched case,
properties of $X_t$ are obtained for a.e. realization.  On the other hand  for the annealed case we
obtain results on the \textit{annealed laws} $$P^*(X(\cdot)\in A)=\int_{C([0,t],C([0,1]))}
P^Z(X(\cdot)\in A)P^C_\mu(dZ),$$ where $P^Z(X(\cdot)\in A)$ denotes the probability law for the
super-Brownian motion  $\{X_t\}$ in the catalyst $\{Z(\cdot)\in C([0,\infty),C([0,1]))\}$.

\subsection{The quenched catalytic OU process}

In order to exhibit the role of the dimension of the underling space we now formulate and prove a
result for the quenched catalytic O-U process on the set $[0,1]^d\subset\mathbb{R}^d$.
\begin{thm}\label{T-Q}
In dimension $d\geq 1$ and with $Z_0\in M_F(\mathbb{R}^d)$, consider the initial value  problem:
\[  \quad \left\lbrace
\begin{array}{ll}\label{cou-01d}
dX(t,x)=A X(t,x)dt+W_{Z_t}(dx,dt)\quad &t\geq0,\quad x\in[0,1]^d\\
X(0,x)=0\quad &x\in \partial[0,1]^d
\end{array}
\right. \] where $A=\frac{1}{2}\Delta$ on $[0,1]^d$ with Dirichlet boundary conditions.  Then  for
almost every realization of $\{Z_t\}$,  $X(t)$ has $\beta$-H\"{o}lder continuous paths in $H_{-n}$
for any $n>d/2$ and $\beta< \frac{1}{2}$, where $H_{-n}$ is the space of distributions on $[0,1]$
defined below in subsection 7.1.
\end{thm}

\subsection{The annealed O-U process with super-Brownian catalyst ($\alpha=2, \beta=1$):\\ State space and sample path continuity}
Now, we apply the techniques developed above to determine some basic properties of the annealed O-U
process $X(\cdot)$ including identification of the state space, sample path continuity and
distribution properties of the random field defined by $X(t)$ for $t>0$.
\medskip
\medskip
\begin{thm}\label{4:delta}
Let $X(t)$ be the solution of the stochastic equation:
\[ (\text{CE})\quad \quad \quad \left\lbrace
\begin{array}{ll}\label{ce}
dX(t,x)=\Delta X(t,x)dt+W_{Z_t}(dt,dx) \quad &t\geq0,\quad x\in\mathbb{R}\\
X(0,x)=0\quad &x\in\mathbb{R}
\end{array}
\right. \]
\medskip
\medskip
\begin{list}{}
\item{(i)} For $t>0$, $X(t)$ is a zero-mean non-Gaussian leptokurtic random field.
\item{(ii)}  if $Z_0=\delta_0(x)$, then the annealed process $X(t)$ satisfies
\[  \mathbb{E}\|X(t)\|^4_{L_2} <\infty\] and has continuous paths in
 $L^2(\mathbb{R})$
\end{list}
\end{thm}

\medskip
\section{Proofs}
In this section we give the proofs of the results formulated in Section \ref{S:properties}.

\subsection{Proof of Theorem \ref{T-Q}.}

\begin{proof}
We first introduce the dual Hilbert spaces $H_n,H_{-n}$.  let $\Delta$ denote the Laplacian on
$[0,1]^d$ with Dirichlet boundary conditions.
 Then $\Delta$ has a CONS of smooth eigenfunctions $\set{\phi_n}$ with
 eigenvalues $\set{\lambda_n}$ which satisfy
 $\sum\limits_j(1+\lambda_j)^{-p}<\infty$ if $p>d/2$, (see \cite{Kress:1}).\\
 Let $E_0$ be the set of $f$ of the form
 $f(x)=\sum\limits_{j=1}^{N}c_j\phi_j(x)$, where the $c_j$
 are constants. For each integer $n$, positive or negative, define
 the space $H_n=\set{ f\in H_0:\norm{f}_n<\infty}$ where the norm is
 given by :
 $$\norm{f}^2_n=\sum\limits_j(1+\lambda_j)^nc_j^2.$$
The $H_{-n}$ is defined to be the dual Hilbert space corresponding to $H_n$.

The solution of $(\ref{cou-01d})$ is given by:
$$X(t,x)=\int_0^t\int_{[0,1]^d}\sum\limits_{k\geq 1}e^{-\lambda_k\,(t-s)}
\phi_k(x)\phi_k(y)W_{Z_s}(dy,ds)$$
 Let
$$A_k(t)=\int_0^t\int_{[0,1]^d}e^{-\lambda_k\,(t-s)}\phi_k(y)W_{Z_s}(dy,ds)$$
Then
$$X(t,x)=\sum\limits_{k\geq1}\phi_k(x)A_k(t)$$
we will show that $X(t)$ has continuous paths on the space $H_n$, which is isomorphic to the set of
formal eigenfunction series
$$f=\sum\limits_{k=1}^{\infty} a_k\phi_k$$
for which
$$\norm{f}_n=\sum a_k^2(1+\lambda_k)^n<\infty$$
Fix some $T>0$, we first find a bound for
$\mathbb{E}\set{\sup\limits_{t\leq T} A_k^2(t)}$, let
$V_k(t)=\int_0^t\int_{[0,1]^d}\phi_k(x)W(Z_s(dx),ds)$, integrating by
parts in the stochastic integral, we obtain:

$$A_k(t)=\int_0^t e^{-\lambda_k(t-s)} dV_k(s)=V_k-\int_0^t\lambda_ke^{-\lambda_k(t-s)}V_k(s)\,ds$$
Thus:
$$\sup\limits_{t\leq T}\abs{A_k(t)}\leq\sup\limits_{t\leq
T}\abs{V_k(t)}(1+\int_0^t\lambda_ke^{\lambda_k(t-s)}ds)
\leq 2 \sup\limits_{t\leq T}\abs{V_k(t)}$$
Hence:
\begin{equation*}
\begin{split}
\mathbb{E}\set{\sup\limits_{t\leq T}A_k^2(t)}&\leq
4\mathbb{E}\set{\sup\limits_{t\leq T}V_k^2(t)}\\
&\leq 16 \mathbb{E}\set{V_k^2(T)}\qquad\qquad\text{(Doob's inequality)}\\
&=16\int_{[0,1]^d}\int_0^T\phi^2_k(x)Z_s(dx)ds\leq 16 T Z_T([0,1]^d)=CT.
\end{split}
\end{equation*}

Therefore: \begin{equation}\label{4:eig} \mathbb{E}\left(\sum\limits_{k\geq 1}\sup\limits_{t\leq T}
A_k^2(t)(1+\lambda_k)^{-n}\right)\leq CT \sum\limits_{k\geq 1}(1+\lambda_k)^{-n}.
\end{equation}
Using now:
\begin{equation*}
\begin{split}
\sum\limits_{k\geq 1}(1+\lambda_k)^{-p}&<\infty\qquad\text{ if }p>d/2
\end{split}
\end{equation*}
then (~\ref{4:eig}) is finite if $n>d/2$ and clearly:
\begin{equation}\label{7.11} E[\norm{X(t)}_{-n}^2]=\sum\limits_{k\geq 1}[A_k^2(t)](1+\lambda_k)^{-n}<\infty\end{equation}
and hence
$X(t)\in H_{-n}$ a.s.. Moreover, if $s>0$,
$$[\norm{X(t+s)-X(t)}_{-n}^2]=\sum E[(A_k(t+s)-A_k(t))^2](1+\lambda_k)^{-n}\leq Cs.$$
Then since conditioned on $Z$, $X(t)$ is Gaussian, (\ref{7.11}) together with  \cite{DpZ 92},
Proposition 3.15, implies that   there is a $\beta$-H\"older continuous version with
$\beta<\frac{1}{2}$.

\end{proof}

\subsection{Second moment measures of SBM}

In this section we evaluate the second moment measures of SBM,
$\mathbb{E}\left(Z_1(dx)Z_2(dy)\right)$, which are needed to  compute $\mathbb{E}[\norm{Z_t}_{L_2}^4]$ in the
the annealed catalytic OU process.

To accomplish this,  given any two random measures
$Z_1$ and $Z_2$, then it is easy to verify that:
\begin{equation*}
\begin{split}
\mathbb{E}(\seq{Z_1,A}),\mathbb{E}(\seq{Z_2,A})&,\quad
A\in\mathcal{B}(\mathbb{R}^d)\\
\mathbb{E}(\seq{Z_1,A}\seq{Z_2,B})&,\quad
A,B\in\mathcal{B}(\mathbb{R}^d)
\end{split}
\end{equation*}
are well defined measures which will be called the first and second moment measures. Similarly,
given the measure-valued process $\{Z_t\}_{t\geq 0}$ one can define the n-th moment measure of n
given random variables denoted by:
$$\mathcal{M}_{t_1\dots t_n}(dx_1,\dots,dx_n)=\mathbb{E}\left(Z_{t_1}(dx_1)Z_{t_2}(dx_2)\cdots Z_{t_n}(dx_n)\right).$$

The main result of this section, is the following:
\begin{prop}  The   super-Brownian motion $Z_t$ in $\mathbb{R}^1$ has  the following first and second moment
measures:
\begin{list}{}
\item{(i)} if $Z_0=dx,\;\mathcal{M}_t(dx)= dx$ the Lebesgue
measure. \item{(ii)} if $Z_0=\delta_0(x)$,then:
\begin{equation*}
\mathcal{M}_t(dx)= p(t,x)\cdot dx
\end{equation*}
\item{(iii)} if $Z_0=dx$:
\begin{equation}
\mathcal{M}_t(dx_1dx_2)=\left(\int_0^tp(2s,x_1,x_2)ds\right)
\cdot\;dx_1\;dx_2
 \end{equation}
\item{(iv)} if  $Z_0=\delta_0(x)$:
\begin{equation}\label{4:meas}
\mathcal{M}_t(dx_1dx_2)=\left(\int_0^t\int_{\mathbb{R}}p(t-s,0,y)
p(2s,x_1,x_2)\;dy\;ds\right) \cdot\;dx_1\;dx_2
 \end{equation}
 where $dx,dx_1dx_2$ denote the Lebesgue measures on $\mathbb{R}\text{
and }\mathbb{R}^2$, respectively.
\item{(v)} if $t_1<t_2$ and $Z_0=dx$:
\begin{equation}\label{5:meas}
\mathcal{M}_{t_1t_2}(dx_1dx_2)=\left(\int_0^{t_1}p(2s+t_2-t_1,x_1,x_2)ds
\right)\cdot\;dx_1\;dx_2
\end{equation}
\item{(vi)} if $t_1<t_2$ and $Z_0=\delta_0(x)$:
\begin{equation}\label{6:meas}
\mathcal{M}_{t_1t_2}(dx_1dx_2)=\left(\int_0^{t_1}\int_{\mathbb{R}}p(t_1-s,0,y)
p(s,y,x_1)p(t_2-t_1+s,y,x_2)\,dyds\right)\cdot\;dx_1\;dx_2
\end{equation}
 where $dx,dx_1dx_2$ denote the Lebesgue measures on $\mathbb{R}\text{
and }\mathbb{R}^2$, respectively.
\end{list}
\end{prop}
\begin{proof}
Consider SBM $Z_t$ with $Z_0=\mu$. Then for
$\lambda\in\mathbb{R}_+$ and $\phi\in C_+(\mathbb{R}^d)$:
\begin{equation*}
\begin{split}
\mathbb{E}_\mu\left[\exp(-\lambda\seq{\phi,Z_t})\right]&=\exp\left[-\seq{u(\lambda,t),\mu}\right]=:
\exp(-F(\lambda))
\end{split}
\end{equation*}
where,  $u(\lambda,t)$ satisfies the equation:
\begin{equation}\label{e:LAP}
\begin{split}
\dot{u}(t)&=\Delta u(t)-u^2(t)\\
u(\lambda,0)&=\lambda\phi
\end{split}
\end{equation}
Then the first moment  is computed according to
$$\mathbb{E}(\int \phi(x)Z_t(dx))=\left[\frac{\partial e^{F(\lambda)}}{\partial \lambda}
\right]_{\lambda=0}=F'(0)\;$$ since $F(0)=0$.
 $F(\lambda)$ will be developed as a Taylor series below. First
 rewrite
(~\ref{e:LAP}) as a Volterra integral equation of the second kind,
namely:
$$u(t)+ \int_0^t T_{t-s}(u^2(s))\;ds=T_t (\lambda\phi) $$

whose solution is given by its Neumann series:

\begin{equation}\label{i:ITER}
u(t)=T_t \,(\lambda \phi) + \sum_{k=1}^n (-1)^k\mathbb{T}^k
(T_t(\lambda\phi))
\end{equation}

where the operators $T_t$ and  $\mathbb{T}$ are defined by:

\begin{equation}
\begin{split}
T_t(\lambda\phi(x))&=\lambda\int_{\mathbb{R}}p(t,x,y)\phi(y)\,dy\\
\mathbb{T}(T_t(\lambda\phi))&=-\int_0^tT_{t-s}[(T_s(\lambda\phi))^2]\;ds=-\lambda^2\int_0^tT_{t-s}[(T_s\phi)^2]\;ds\\
&=-\lambda^2\int_0^t\int_{\mathbb{R}}p(t-s,x,y)
\left(\int_{\mathbb{R}}p(s,y,z)\phi(z)\,dz\right)^2\,dy\,ds.
\end{split}
\end{equation}

We obtain $F(\lambda)= \sum_{n=1}^\infty
(-1)^{n+1}c_n\,\lambda^n$, in particular:
\begin{equation}\label{1:momt}
c_1=\int_\mathbb{R}T_t\phi\,\mu(dx)=\int_\mathbb{R}\int_\mathbb{R}p(t,x,y)\phi(y)\,dy\,\mu(dx)
\end{equation}
\begin{equation}\label{2:momt}
\begin{split}
c_2&=\int_\mathbb{R} \int_0^t\int_{\mathbb{R}}p(t-s,x,y)
\left(\int_{\mathbb{R}}p(s,y,z)\phi(z)\,dz\right)^2
\,dy\,ds\,\mu(dx)\\
&=\int_0^t\int_{\mathbb{R}^4}p(t-s,x,y)p(s,y,z_1)p(s,y,z_1)\phi(z_1)
\phi(z_2)\,dz_1\,dz_2\,dy\,\mu(dx)\,ds
\end{split}
\end{equation}
\begin{equation}\label{3:momt}
\begin{split}
c_4&=\int_\mathbb{R} \int_0^tp(t-s,x,w)\\
&\;\;\left[\int_0^s\int_{\mathbb{R}}p(s-s_1,w,y)\left(\int_{\mathbb{R}}p(s,y,z)
\phi(z)\,dz\right)^2\,dy\,ds_1\right]^2\,ds\,dw\mu(dx)
\end{split}
\end{equation}
and so on.

Taking now $d=1$, $Z_0=dx$, a given Borel set
$A\in\mathcal{B}(\mathbb{R}) \text{ and }\phi=\mathbb{I}_A$, we
obtain from (~\ref{1:momt}):

\begin{equation*}
\begin{split}
\mathbb{E}\seq{A,Z_t}&=\int_\mathbb{R}\int_\mathbb{R}
p(t,x,y)\phi(y)dy\,dx=\int_\mathbb{R}\int_\mathbb{R}p(t,x,y)\,dx\,\mathbb{I}_A(y)\,dy\\
&=\int_Ady
\end{split}
\end{equation*}
i.e. $\mathbb{E}\{Z_t(dx)\}$ is the Lebesgue measure.\\
\medskip
The covariance measure $\mathbb{E}(Z_t(dx_1)Z_t(dx_2))$ can now be
computed applying the same procedure to
$\phi=\lambda_1\mathbb{I}_{A_1}+\lambda_2\mathbb{I}_{A_2}$ in
which case we conclude that $
-\seq{u(t),\mu}=F(\lambda_1,\lambda_2)$ is again a Taylor series
in $\lambda_1,\text{ and } \lambda _2$, with constant coefficient
equal to zero, and only the coefficient of $\lambda_1\lambda_2$
has to be computed:
\begin{equation}
\begin{split}
\mathbb{T}( T_t\phi)&=-\int_0^tT_{t-s}[(T_s\phi)^2]\;ds\\
&=-\int_0^tT_{t-s}[\lambda_1^2(T_s\mathbb{I}_{A_1})^2+
2\lambda_1\lambda_2T_s\mathbb{I}_{A_1}T_s\mathbb{I}_{A_2}+
\lambda_2^2(T_s\mathbb{I}_{A_2})^2] \,ds\\
&=-\lambda_1^2\int_0^tT_{t-s}[(T_s\mathbb{I}_{A_1})^2]\,ds
-2\lambda_1\lambda_2\int_0^tT_{t-s}[T_s\mathbb{I}_{A_1}\cdot
T_s\mathbb{I}_{A_2}]ds\\
&\;\;\;-\lambda_2^2\int_0^tT_{t-s}[(T_s\mathbb{I}_{A_2})^2]\,ds
\end{split}
\end{equation}

So:
\begin{equation}\label{4:secm}
\begin{split}
\mathbb{E}(Z_t(A_1)Z_t(A_2))&=\frac{1}{2}\left[\frac{\partial^2
F(\lambda_1,\lambda_2)}{\partial\lambda_1\partial\lambda_2}
\right]_{\lambda_1=\lambda_2=0}\\
&=\int_0^t\int_{\mathbb{R}}T_{t-s}[T_s\mathbb{I}_{A_1}\cdot
T_s\mathbb{I}_{A_2}]\,\mu(dx)\,ds.
\end{split}
\end{equation}
In order to obtain the density for the measure
$\mathbb{E}(Z_t(dx_1)Z_t(dx_2))$, we write in detail the last
integral as follows:
$$\int_0^t\int_{\mathbb{R}}\int_{\mathbb{R}}p(t-s,x,y)
\left(\int_{A_1}p(s,y,z_1)dz_1\cdot\int_{A_2}p(s,y,dz_2)dz_2
\right)\mu(dx)\,ds\,dy$$ applying Fubini yields:
$$\int_{A_1}\!\!\int_{A_2}\!\!\int_0^t\int_{\mathbb{R}}\int_{\mathbb{R}}
p(t-s,x,y)p(s,y,z_1)p(s,y,z_2)\;dy\;\mu(dx)\;ds\;dz_1\;dz_2$$ from
which, we conclude that $\mathbb{E}(Z_t(dx_1)Z_t(dx_2))$ has the
following density w.r.t. Lebesgue measure:
\begin{equation}\label{4:denst}
\int_0^t\int_{\mathbb{R}}\int_{\mathbb{R}}
p(t-s,x,y)p(s,y,z_1)p(s,y,z_2)\;dy\;\mu(dx)\;ds
\end{equation}
assuming $\mu(dx)$ is the Lebesgue measure in $\mathbb{R}$ and
$dx_1\cdot d x_2$ is the Lebesgue measure in $\mathbb{R}^2$ one
obtains the second moment measure $\mathcal{M}_t(dx_1dx_2)$
defined as:
\begin{equation}
\mathcal{M}_t(dx_1dx_2)\doteq\left(\int_0^tp(2s,x_1,x_2)ds\right)
\cdot\;dx_1\;dx_2
\end{equation}

For the  case $t_1\neq t_2$, assume $t_1<t_2$ then:
\begin{equation*}
\begin{split}
&\mathbb{E}\left[\int\phi_1Z_{t_1}\cdot\int\phi_2Z_{t_2}\right]=\\
&=\mathbb{E}\left[\mathbb{E}\left[\int\phi_1Z_{t_1}\cdot\int\phi_2Z_{t_2}\right]
|\sigma(Z_r):0\leq r \leq t_1\right]\\
&=\mathbb{E}\left[\int\phi_1Z_{t_1}\cdot\mathbb{E}\left[\int\phi_2Z_{t_2}
|\mathcal{F}_{Z_{t_1}}\right]\right]\\
&=\mathbb{E}\left[\int\phi_1Z_{t_1}\cdot
T_{t_2-t_1}\left[\int\phi_2Z_{t_1}\right]\right]\\
&=\mathbb{E}\left[\int\phi_1Z_{t_1}\cdot\int
T_{t_2-t_1}(\phi_2)Z_{t_1}\right]\\
&=\mathbb{E}[\seq{\phi_1,Z_{t_1}}\seq{T_{t_2-t_1}(\phi_2),Z_{t_1}}]
\end{split}
\end{equation*}
so, replacing in (~\ref{4:secm}) $\mathbb{I}_{A_2}\text{ by }
T_{t_2-t_1}(\mathbb{I}_{A_2})$ and performing the same analysis,
we can define following measure on $\mathbb{R}^2$
\begin{equation*}
\mathcal{M}_{t_1t_2}(dx_1dx_2)=\left(\int_0^{t_1}p(2s+t_2-t_1,x_1,x_2)ds \right)\cdot dx_1dx_2.
\end{equation*}
Finally, note that  if $Z_0=\delta_0(x)$, then (\ref{4:denst}) yields:
\begin{equation*}
\mathcal{M}_{t_1t_2}(dx_1dx_2)=\left(\int_0^{t_1}\int_{\mathbb{R}}
p(t_1-s,0,y)p(s,y,x_1)p(t_2-t_1+s,y,x_2)\,dy\,ds \right)\cdot dx_1dx_2
\end{equation*}
\end{proof}

\begin{rem}
The above procedure can also be used for any $\alpha\in(0,2],\beta=1$, any dimension $d\geq 1 $ and
any $C_0$ semigroup $S_t$ with probability transition function $p(t,x,y)$.
\end{rem}

\subsection{Proof of Theorem \ref{4:delta}}

 We can now use the  results of the last subsection to
determine properties  of the solutions of (\ref{SPCOU}).

(i) It follows from the form of the characteristic function (given by (\ref{CFE}) with $\beta =1$
in (\ref{2:redif})) that the log of the characteristic function of $\langle \lambda\phi,X_t\rangle$
is not a quadratic in $\lambda$ and in fact the fourth cumulant is positive so that the random
field is leptokurtic.

(ii)  Recall that
$$X(t,x)=\int_0^t\int_{\mathbb{R}}p(t-s,x,y)W_{Z_s}(ds,dy)$$

\smallskip
\begin{lem}
\begin{equation}\label{eqnlema}
\begin{split}
\mathbb{E}[\norm{X_t}_{L_2}^4]&=C\mathbb{E}\left[\int_0^t\int_\mathbb{R}\frac{Z_s(dy)}{\sqrt{t-s}}ds\right]^2\\
&+2\int_{\mathcal{D}_2}p^2(2t-s_1-s_2,y_1,y_2)
\,\mathbb{E}\left(Z_{s_1}(dy_1)Z_{s_2}(dy_2)\right)\;ds_1ds_2.
\end{split}
\end{equation}

\end{lem}
\begin{proof}
Using the shorthand $p=p(t-s,x,y),\,dW^s=W_{Z_s}(ds,dy);\,
 p_i(x)=p(t-s_i,x,y_i),\,p_{ij}=p(t-s_i,x_j,y_i),\,
dW^{s_i}=W_{Z_{s_i}}(ds_i,dy_i),\, \text{for }, i=1,\cdots 4,j=1,2 \text{ as well as } {\cal
D}_1=[0,t]\times\mathbb{R},\, {\cal D}_2=[0,t]^2\times\mathbb{R}^2,\,{\cal D}_4=[0,t]^4\times
\mathbb{R}^4$, and noting that $p_{ij}=p_i(x_j)$, we first compute:
\smallskip
\begin{equation*}
\begin{split}
\mathbb{E}[\norm{X_t}_{L_2}^4|Z_s,0\leq s\leq t]&=\mathbb{E}\left[\int_{\mathbb{R}}\left(\int_0^t\int_\mathbb{R}
p(t-s,x,y)W_{Z_s}(ds,dy)\right)^2dx\right]^2\\
&=\mathbb{E}\left[\int_\mathbb{R}\left(\int_{\mathcal{D}_1}\!p\, dW^s\right)^2dx_1\right]
\left[\int_\mathbb{R}\left(\int_{\mathcal{D}_1}\!p\,
dW^s\right)^2dx_2\right]\\
&=\mathbb{E}\left[\int_\mathbb{R}\left(\int_{\mathcal{D}_1}\!
p_1(x_1)\,dW^{s_1}\cdot\int_{\mathcal{D}_1}\!
p_2(x_1)\,dW^{s_2}\right)dx_1\right]\\
&\quad\quad\cdot\left[\int_\mathbb{R}\left(\int_{\mathcal{D}_1}
\!p_3(x_2)\,dW^{s_3}\cdot\int_{\mathcal{D}_1}\!
p_4(x_2)\,dW^{s_4}\right)dx_2\right]\\
&=\mathbb{E}\left[\int_\mathbb{R}\int_{\mathcal{D}_2}\! p_{11}p_{21}\,dW^{s_1}dW^{s_2}\,dx_1\right]
\left[\int_\mathbb{R}\int_{\mathcal{D}_2}\!
p_{32}p_{42}\,dW^{s_3}dW^{s_4}\,dx_2\right]\\
&=\mathbb{E}\left[\int_{\mathbb{R}^2}\int_{\mathcal{D}_4} p_{11}p_{21}p_{32}p_{42}\,dW^{s_1}
dW^{s_2}dW^{s_3}dW^{s_4}
dx_1dx_2\right]\\
&=\int_{\mathbb{R}^2}\int_{\mathcal{D}_4} p_{11}p_{21}p_{32}p_{42}\,\mathbb{E}(dW^{s_1}
dW^{s_2}dW^{s_3}
dW^{s_4})dx_1dx_2\\
&=\mathbb{I}_1+\mathbb{I}_2+\mathbb{I}_3.
\end{split}
\end{equation*}
\medskip
Each one of the above integrals can be evaluated using the independence of the increments of the
Wiener process, according to the following cases:
\\
\medskip
\medskip

\textbf{Case 1:} $ s_1=s_2\,\wedge\, s_3=s_4$, then:
$p_{11}p_{21}=p_{11}^2=p_1^2(x_1),\,p_{32}p_{42}=p_{32}^2=p_3^2(x_2)$,
 and $\mathbb{E}(dW^{s_1} dW^{s_2}dW^{s_3}dW^{s_4})=Z_{s_1}(dy_1)
 ds_1Z_{s_3}(dy_3)ds_3$, hence:\\
 \begin{equation*}
 \begin{split}
 \mathbb{I}_1&=\int_{\mathbb{R}^2}\int_{\mathcal{D}_2}
 p_1^2(x_1)p_3^2(x_2)\,Z_{s_1}(dy_1)ds_1Z_{s_3}(dy_3)ds_3dx_1dx_2\\
 &=\left[\int_\mathbb{R}\int_{\mathcal{D}_1}p_1^2(x_1)\,Z_{s_1}(dy_1)ds_1
 dx_1\right]
 \left[\int_\mathbb{R}\int_{\mathcal{D}_1}p_3^2(x_2)\,Z_{s_3}(dy_3)ds_3
 dx_2\right]\\
 &=\left[\int_\mathbb{R}\int_{\mathcal{D}_1}p^2\,Z_s(dy)dsdx\right]^2\\
 &=\left[\int_0^t\int_\mathbb{R}\int_{\mathbb{R}}p^2\,dx\,Z_s(dy)ds
 \right]^2\\
 &=C\left[\int_0^t\int_\mathbb{R}\frac{Z_s(dy)}{\sqrt{t-s}}ds\right]^2
 \end{split}
 \end{equation*}
 \textbf{Case 2: }$s_1=s_3\wedge s_2=s_4$, then:
 $p_{11}p_{32}=p_{11}p_{12}=p_1(x_1)p_1(x_2),\,p_{21}p_{42}=p_{21}p_{22}
 =p_2(x_1)p_2(x_2)$,
  and $\mathbb{E}(dW^{s_1} dW^{s_2}dW^{s_3}dW^{s_4})=Z_{s_1}(dy_1)
  ds_1Z_{s_2}(dy_2)ds_2$, hence:
\begin{equation*}
\begin{split}
\mathbb{I}_2&=\int_{\mathbb{R}^2}\int_{\mathcal{D}_2}
p_1(x_1)p_1(x_2)p_2(x_1)p_2(x_2)\,Z_{s_1}(dy_1)ds_1Z_{s_2}(dy_2)ds_2dx_1
dx_2\\
&=\int_{\mathcal{D}_2}\left(\int_\mathbb{R}p_1(x_1)p_2(x_1)\,dx_1\cdot
\int_\mathbb{R}p_1(x_2)p_2(x_2)\,dx_2\right)
\,Z_{s_1}(dy_1)ds_1Z_{s_2}(dy_2)ds_2dx_1\\
&=\int_{\mathcal{D}_2}p^2(2t-s_1-s_2,y_1,y_2) \,Z_{s_1}(dy_1)ds_1Z_{s_2}(dy_2)ds_2
\end{split}
\end{equation*}
\textbf{Case 3: }$s_1=s_4\wedge s_2=s_3$, then:
 $p_{11}p_{42}=p_{11}p_{12}=p_1(x_1)p_1(x_2),\,p_{21}p_{32}=p_{21}p_{22}
 =p_2(x_1)p_2(x_2)$,
  and $\mathbb{E}(dW^{s_1} dW^{s_2}dW^{s_3}dW^{s_4})=Z_{s_1}(dy_1)
  ds_1Z_{s_2}(dy_2)ds_2$, and we get the same result as before,
  namely:
\begin{equation*}
\mathbb{I}_3=\int_{\mathcal{D}_2}p^2(2t-s_1-s_2,y_1,y_2) \,Z_{s_1}(dy_1)ds_1Z_{s_2}(dy_2)ds_2
\end{equation*}

 Putting everything together, yields:
\begin{equation}
\begin{split}
&\mathbb{E}[\norm{X_t}_{L_2}^4|Z_s,0\leq s\leq t]\\
&=C\left[\int_0^t\int_\mathbb{R}\frac{Z_s(dy)}{\sqrt{t-s}}ds\right]^2+2\int_{\mathcal{D}_2}p^2(2t-s_1-s_2,y_1,y_2) \,Z_{s_1}(dy_1)ds_1Z_{s_2}(dy_2)ds_2
\end{split}
\end{equation}
So:
\begin{equation}
\begin{split}
\mathbb{E}[\norm{X_t}_{L_2}^4]&=\mathbb{E}\left(
\mathbb{E}[\norm{X_t}_{L_2}^4|Z_s,0\leq s\leq t]\right)\\
&=C\mathbb{E}\left[\int_0^t\int_\mathbb{R}\frac{Z_s(dy)}{\sqrt{t-s}}ds\right]^2+2\int_{\mathcal{D}_2}p^2(2t-s_1-s_2,y_1,y_2)
\,\mathbb{E}\left(Z_{s_1}(dy_1)Z_{s_2}(dy_2)\right)\;ds_1ds_2.
\end{split}
\end{equation}
\end{proof}
\medskip
\medskip

We now return to the proof of Theorem \ref{4:delta}
\begin{proof}
\noindent The first term on the right of \ref{eqnlema} is evaluated below
\begin{equation}\label{1:expct}
\begin{split}
&\mathbb{E}\left[\int_0^t\int_\mathbb{R}\frac{Z_s(dx)}{\sqrt{t-s}}ds\right]^2=\\
&=\mathbb{E}\left[\int_0^t\int_0^t\left(\frac{1}{\sqrt{(t-s_1)(t-s_2)}}
\int_\mathbb{R}\int_\mathbb{R}Z_{s_1}(dx_1)Z_{s_2}(dx_2)\right)ds_1\,ds_2\right]\\
&=\int_0^t\int_0^{s_2}\left(\frac{1}{[(t-s_1)(t-s_2)]^{1/2}}
\int_\mathbb{R}\int_\mathbb{R}\mathbb{E}[Z_{s_1}(dx_1)Z_{s_2}(dx_2)]\right)
ds_1\,ds_2\\
&+\int_0^t\int_{s_2}^t\left(\frac{1}{[(t-s_1)(t-s_2)]^{1/2}}
\int_\mathbb{R}\int_\mathbb{R}\mathbb{E}[Z_{s_1}(dx_1)Z_{s_2}(dx_2)]\right)
ds_1\,ds_2\\
\end{split}
\end{equation}

Assume now $Z_0=\delta_0(x)$, and $s_1\leq s_2$, using (~\ref{6:meas}), one obtains:
\begin{equation*}
\int_\mathbb{R}\int_{\mathbb{R}}\mathbb{E}[Z_{s_1}(dx_1)Z_{s_2}(dx_2)]
=\int_0^{s_1}\int_\mathbb{R}p(s_1-s,0,y)\,dy\,ds\\
=\int_0^{s_1}ds=s_1
\end{equation*}
and similarly, for $s_2\leq s_1$ :
$$\int_\mathbb{R}\int_{\mathbb{R}}\mathbb{E}[Z_{s_1}(dx_1)Z_{s_2}(dx_2)]=s_2$$
so that (~\ref{1:expct}) can be written as:
\begin{equation*}
\mathbb{E}\left[\int_0^t\int_\mathbb{R}\frac{Z_s(dx)}{\sqrt{t-s}}ds\right]^2
=\int_0^t\int_0^{s_2}\frac{s_1ds_1ds_2}{[(t-s_1)(t-s_2)]^{1/2}}
+\int_0^t\int_{s_2}^t\frac{s_2ds_1ds_2}{[(t-s_1)(t-s_2)]^{1/2}}
\end{equation*}
the first of the above integrals one can be directly done:
\begin{equation*}
\begin{split}
\frac{1}{(t-s_2)^{1/2}}\int_0^{s_2}\frac{s_1\;ds_1}{(t-s_1)^{1/2}}&=\frac{1}{(t-s_2)^{1/2}}\left(\frac{4}{3}t^{3/2}+\frac{2}{3}
(t-s_2)^{3/2}-2t(t-s_2)^{1/2}\right)\\
&\leq\frac{1}{(t-s_2)^{1/2}}\left(\frac{4}{3}t^{3/2}+\frac{2}{3}(t-s_2)^{3/2}
\right)\\
&=\frac{4}{3}\frac{t^{3/2}}{(t-s_2)^{1/2}}+\frac{2}{3}.
\end{split}
\end{equation*}
Hence:
$$\int_0^t\int_0^{s_2}\frac{s_1ds_1ds_2}{[(t-s_1)(t-s_2)]^{1/2}}\leq
\frac{4}{3}t^2+\frac{2}{3}t^2=2t^2$$ similarly, for the second integral:
$$\frac{s_2}{(t-s_2)^{1/2}}\int_{s_2}^t\frac{ds_1}{(t-s_1)^{1/2}}=s_2$$
and:
$$\int_0^t\int_{s_2}^t\frac{s_2ds_1ds_2}{[(t-s_1)(t-s_2)]^{1/2}}\leq\frac{t^2}
{2}.$$ Both inequalities together imply:
\begin{equation}\label{1:ineq}
\mathbb{E}\left[\int_0^t\int_\mathbb{R}\frac{Z_s(dx)}{\sqrt{t-s}}ds\right]^2
\leq C_1t^2
\end{equation}
\medskip
Similarly, we estimate below the second integral of
(~\ref{4:nob}):
\begin{equation*}
\begin{split}
&\int_{\mathcal{D}_2}p^2(2t-s_1-s_2,x_1,x_2)
\,\mathbb{E}\left(Z_{s_1}(dx_1)Z_{s_2}(dx_2)\right)\;ds_1ds_2=\\
&=\int_0^t\int_0^{s_2}\left(\int_\mathbb{R}\int_\mathbb{R}
p^2(2t-s_1-s_2,x_1,x_2)\mathbb{E}\left(Z_{s_1}(dx_1)Z_{s_2}(dx_2)\right)
\right)ds_1ds_2\\
&+\int_0^t\int_{s_2}^t\left(\int_\mathbb{R}\int_\mathbb{R}
p^2(2t-s_1-s_2,x_1,x_2)\mathbb{E}\left(Z_{s_1}(dx_1)Z_{s_2}(dx_2)\right)
\right)ds_1ds_2
\end{split}
\end{equation*}
when $0\leq s_1\leq s_2$. Using (~\ref{6:meas}) with $0\leq s\leq
s_1\leq s_2\leq t$, we obtain the following estimate for:
\begin{equation*}
\begin{split}
&\widetilde{\mathbb{I}}_1(t,s_1,s_2)\doteq\int_\mathbb{R}\int_\mathbb{R}
p^2(2t-s_1-s_2,x_1,x_2)\mathbb{E}\left(Z_{s_1}(dx_1)Z_{s_2}(dx_2)\right)\\
&=\frac{\sqrt{2}}{\sqrt{2t-s_1-s_2}}\int_0^{s_2}
\!\!\!\frac{1}{\sqrt{(2t-s_1-s_2)+4s}}ds\\
&\leq\frac{\sqrt{2}}{\sqrt{2t-s_1-s_2}}\int_0^{s_2}
\!\!\!\frac{1}{\sqrt{4s}}ds\\
&\leq C_1\sqrt{\frac{s_2}{2t-s_1-s_2}}
\end{split}
\end{equation*}
\medskip
\medskip
so, we obtain:
\begin{equation}\label{2:ineq}
\begin{split}
\mathbb{I}_1=\int_0^t\int_0^{s_2}\widetilde{\mathbb{I}}_1(t,s_1,s_2)\;ds_1ds_2&\leq
C_1\int_0^t\int_0^{s_2}
\sqrt{\frac{s_2}{2t-s_1-s_2}}ds_1\;ds_2\\
&=C_1\int_0^t\sqrt{s_2}\left[2\sqrt{2t-s_2}-2\sqrt{2t-2s_2}\right]ds_2\\
&\leq C_2\int_0^t\sqrt{s_2}\sqrt{t}ds_2\leq C_3 t^2.
\end{split}
\end{equation}
When $0\leq s_2\leq s_1$, using (~\ref{6:meas}) with $0\leq s\leq s_2\leq s_1\leq t$, and following
the same steps above for the integral:
\begin{equation*}
\begin{split}
\widetilde{\mathbb{I}}_2&\doteq\int_\mathbb{R}\int_\mathbb{R}
p^2(2t-s_1-s_2,x_1,x_2)\mathbb{E}\left(Z_{s_2}(dx_1)Z_{s_1}(dx_2)\right)\\
&\leq C_1\sqrt{\frac{t}{2t-s_1-s_2}}
\end{split}
\end{equation*}
therefore:
\begin{equation}\label{3:ineq}
\begin{split}
\mathbb{I}_2=\int_0^t\int_0^{s_2}\widetilde{\mathbb{I}}_2(t,s_1,s_2)\;ds_1ds_2&\leq
C_4\int_0^t\int_{s_2}^t
\sqrt{\frac{t}{2t-s_1-s_2}}ds_1\;ds_2\\
&=C_5\int_0^t\sqrt{t}\left[2\sqrt{2t-2s_2}-2\sqrt{t-s_2}\right]ds_2\\
&\leq C_6\int_0^t\sqrt{t}\sqrt{t}ds_2=C_6 t^2.
\end{split}
\end{equation}
Gathering (~\ref{1:ineq}), (~\ref{2:ineq}) and (~\ref{3:ineq})
yields:
$$\mathbb{E}[\norm{X_t}_{L_2}^4]\leq Ct^2.$$
In the same way we can obtain the estimates
$$\mathbb{E}[\norm{X_t-X_s}_{L_2}^4]\leq C(t-s)^2.$$
for the increments which shows the continuity of the paths using
Kolmogorov's criteria.
\end{proof}

\section{Comments and Open Problems}
\begin{enumerate}
\item Our results corrrespond to the analogue of the  Heston model (HM) with $\rho =0$. The general case when $\rho\ne 0$ would require
additional techniques and is left as an open problem.

\item The study of the properties of the annealed case for arbitrary $\alpha,\beta\ne 1$ and for the more general continuous state branching involves a catalyst with infinite second
moments and is left as an open problem.

\item It would be interesting to determine the state space for annealed process in $\mathbb{R}^d$ with $d>1$.

\end{enumerate}

\end{document}